\documentclass[preprint,12pt]{elsarticle}
\usepackage{}

\usepackage{amsfonts}
\usepackage{amsmath}
\usepackage{amssymb}
\usepackage{graphics}
\usepackage{graphicx}
\usepackage{mathrsfs}
\usepackage{indentfirst}
\usepackage{stmaryrd}
\usepackage{latexsym}
\usepackage{xypic}
\newtheorem{definition}{Definition}[section]

\newtheorem{theorem}[definition]{Theorem}
\newtheorem{proposition}[definition]{Proposition}
\newtheorem{example}[definition]{Example}

\newtheorem{remark}[definition]{Remark}
\newtheorem{corollary}[definition]{Corollary}
\newproof{proof}{\textbf{Proof}}

\journal{}

\begin{document}

\begin{frontmatter}


% use the tnoteref command within \title for footnotes;
% use the tnotetext command for theassociated footnote;
% use the fnref command within \author or \address for footnotes;
% use the fntext command for theassociated footnote;
% use the corref command within \author for corresponding author footnotes;
% use the cortext command for theassociated footnote;
% \author{Name\corref{cor1}\fnref{label2}}
% \address{Address\fnref{label3}}

\title{\textbf{On derivations of MV-algebras}}


\author{Jun Tao Wang$^{a,*}$, Bijan Davvaz$^b$, Peng Fei He$^c$}
\cortext[cor1]{Corresponding author. \\
Email addresses: wjt@stumail.nwu.edu.cn(J.T. Wang), davvaz@yaz.ac.ir (B. Davvaz),\\ hepengf1986@126.com (P.F. He).}
%skywine@gmail.com(Y.B. Jun),

\address[A]{School of Mathematics, Northwest University, Xi'an, 710127, China}
%\address[B]{Department of Mathematics Education, Gyeongsang National University, Jinju 660-701, Korea}
\address[B]{Department of Pure Mathematics, Yazd University, Yazd 89195-741, Iran}

\begin{abstract}
In this paper, we investigate related properties of some particular derivations  and give some characterizations of additive derivations in MV-algebras. Then, we obtain that the fixed point set of additive derivations is still an MV-algebra. Also, we study Boolean additive derivations and their adjoint derivations. In particular, we get that the fixed point set of Boolean addition derivations and that of their adjoint derivations are isomorphism. Moreover, we prove that every MV-algebra
is isomorphic to the direct product of  the fixed point set of Boolean additive derivations and that of their adjoint derivations. Finally, we show that the structure of a Boolean algebra is completely determined by its set of all Boolean additive (implicative) derivations.
 %In the paper, we introduce some particular stabilizers and investigate related properties of them in MTL-algebras.
 %Then, we also characterize some special classes of MTL-algebras, for example, IMTL-algebras, integral MTL-algebras, G\"{o}del algebras and MV-algebras, in terms of these stabilizers.
 %Moreover, we discuss the relation between stabilizers and several other special filters (lattice ideals) in MTL-algebras.
 %Finally, we discuss the relation between these stabilizers and prove that the right implicative stabilizer and right multiplicative stabilizer are order isomorphic. This results also give answers to show that serval open problems, which was proposed by Motamed et al, in A new class of BL-algebras [Soft Comput, {\bf 21}(2017),686-693].
% Text of abstract
\end{abstract}

\begin{keyword} MV-algebra; additive derivation; fixed point set; Boolean algebra

% keywords here, in the form: keyword \sep keyword stabilizer


% MSC codes here, in the form: \MSC code \sep code
%\MSC[2010] 06D35 \sep 06D99 \sep 03G25%(2000 is the default)
%\MSC 06B35 \sep 06B99
\end{keyword}

\end{frontmatter}

% main text
\section{Introduction}
\label{intro} It is well known that certain information processing approaches, especially inferences based on certain information, are based on the classical logic (classical two-valued logic). Naturally, it is necessary to establish some rational logic systems as a logical foundation for uncertain information processing. For this reason, various types of non-classical logic systems have been proposed and researched. In recent years, non-classical logic has become a formal and useful tool in computer science for dealing with uncertain and fuzzy information. Various logical algebras have been proposed as semantic systems for non-classical logic systems. Among these logical algebras, MV-algebras are the first class of logic algebras introduced and investigated. In 1958, C.C.Chang introduced the notion of MV-algebras for the purpose of providing algebraic proof of the completeness theorem of infinite-valued propositional logics \cite{Chang1}. We are speaking here of the infinite-valued logic proposed in 1930 by {\L}ukasiewicz and Tarski \cite{Tarski} with truth values in the internal [0,1] of real numbers. Thus, in a certain sense, MV-algebras stand in relations to many-valued logic as Boolean algebras do to classical two-valued logic. Furthermore, Chang \cite{Chang2} established a bijective correspondence between the linearly ordered MV-algebras and the linearly ordered abelian $\ell$-groups with strong unit, then used this result in order to obtain an algebraic proof for the completeness theorem of {\L}ukasiewicz propositional logic in another way. In \cite{Mundici1}, Mundici extended Chang's result by proving the categorial equivalence between MV-algebras and abelian $\ell$-groups with strong unit.

The notion of derivations, introduced from the analytic theory, is helpful for studying algebraic structures and properties in algebraic systems. In fact, the notion of derivation in ring theory is quite old and plays a significant role in algebraic geometry. In 1957, Posner\cite{Posner} introduced the notion of a derivations in a prime ring. After that a number of research articles have been appeared on derivations in the theory of rings and references there in \cite{Bell1,Bell2,Davvaz1,Albas}. Inspired by derivations on rings, Jun et al \cite{Jun} applied the notion  of derivations to BCI-algebras and gave some characterizations of p-semisimple BCI-algebras. Based on this, several authors have studied derivations in BCI-algebras \cite{Zhan,Davvaz2,Borzooei}. In 2010, Alshehri\cite{Alshehri} applied the notions of (additive) derivations to MV-algebras and discussed some related properties, they also proved that an additive derivation of a linearly ordered MV-algebra is isotone. After the work of Alshehri, many research articles have appeared on the derivations of MV-algebras in different aspects \cite{Yazarli,Ghorbain,Ardekani}, for example, Yazarli et al \cite{Yazarli} further investigated several kinds of generalized derivations on MV-algebras and obtain some interesting results.  Ardekani and Davvaz \cite{Ardekani} introduced the notion of $f$-derivations and $(f,g)$-derivations of MV -algebras and investigated some related properties of them. Ghorbain et al \cite{Ghorbain} introduced several kinds of these derivations and discuss some related results. They also discuss the relationship between these new derivations on MV-algebras. Recently, the notion of derivations has been extended to various logical algebras such as BL-algebras \cite{Xin}, residuated lattices \cite{he} and their non-commutative cases.

This paper is a continuation of the research from \cite{Alshehri}. The paper is organized as follows: In Section 2, we review some basic definitions and results about MV-algebras. In Section 3, we characterize some particular derivations in MV-algebras. In Section 4, we study Boolean derivations and their adjoint derivations. In particular, we show that every MV-algebra is isomorphic to the direct product of  the fixed point set of Boolean additive derivations and that of their adjoint derivations.

\section{Preliminaries}
In this section, we summarize some definitions and results about MV-algebras, which will be used in the following sections.
\begin{definition}\emph{\cite{Chang1} An algebra $(L,\oplus,\ast,0)$ of type $(2,1,0)$ is called an \emph{MV-algebra} if it satisfies the following conditions:}
\begin{enumerate}[(MV1)]
  \item \emph{$(L,\oplus,0)$ is a commutative monoid,
  \item $(x^\ast)^\ast=x$,
  \item $0^\ast\oplus x=0^\ast$,
  \item $(x^\ast\oplus y)^\ast\oplus y=(y^\ast\oplus x)^\ast\oplus x$, for any $x,y\in L$.}
\end{enumerate}
\end{definition}

In what follows, by $L$ we denote the universe of an MV-algebra $(L,\oplus,\ast,0)$.  On each MV-algebra $L$, we define the constant $1$ and the operations $\odot$, $\rightarrow$ as follows: $1=0^\ast$, $x\odot y=(x^\ast\oplus y^\ast)^\ast$ and $x\rightarrow y=x^\ast\oplus y$. We define $x\leq y$ if and only if $x^\ast\oplus y=1$. It follows that $\leq$ is a partial order, called the \emph{natural order} of $L$. On each MV-algebra, the natural order determines a lattices structure, in which, $x\vee y=(x\odot y^\ast)\oplus y$, $x\wedge y=x\odot(x^\ast\oplus y)$. In fact, one can prove that $(L,\wedge,\vee,0,1)$ is a distributive lattice. Since MV-algebras form a variety, the notions of homomorphism, subalgebra are just the particular cases of the corresponding universal algebraic notions \cite{Chang1,Chang2,Tarski}.

\begin{example}\emph{\cite{Mundici2}  Let $L=[0,1]$ be the real unit interval. If we define $x\oplus y=\min\{1,x+y\}$ and $x^\ast=1-x$ for any $x,y\in L$, then $(L,\oplus,\ast, 0)$ is an MV-algebra. Also, for each number $n\geq 2$, then $n$-element set $S_n=\{0,\frac{1}{n-1},\frac{2}{n-2},\cdots,\frac{n-1}{n-2},1\}$ is a subalgebra of $L$.}
\end{example}

\begin{proposition}\emph{\cite{Chang1,Chang2,Tarski} In an MV-algebra, the following properties hold:}
\begin{enumerate}[(1)]
  \item \emph{$x\oplus x^\ast=1$,
  \item $x\odot x^\ast=0$,
  \item $x\leq y$ if and only if $x\rightarrow y=1$,
  \item $x\odot y\leq x\wedge y$,
  \item $x\rightarrow (y\wedge z)=(x\rightarrow y)\wedge(x\rightarrow z)$,
  \item $(x\vee y)\rightarrow z=(x\rightarrow z)\wedge(y\rightarrow z)$,
  \item $x\leq y$ implies $x\odot z\leq y\odot z$,
  \item $x\vee y=(x\rightarrow y)\rightarrow y=(y\rightarrow x)\rightarrow x$,
  \item $x\leq y\rightarrow x$,
  \item $x\odot(y\vee z)=(x\odot y)\vee(x\odot z)$,
  \item $x\oplus(y\wedge z)=(x\oplus y)\wedge(x\oplus z)$, for all $x,y,z\in L$.}
\end{enumerate}
\end{proposition}

MV-algebras are non-idempotent generalizations of Boolean algebras. Indeed, Boolean algebras are just the MV-algebras obeying the additional identity $x\oplus x=x$ or $x\odot x=x$. Let $L$ be an MV-algebra and $B(L)=\{a\in L|a\oplus a=a\}=\{a\in L|a\odot a=a\}$ be the set of all idempotent elements of $L$. Then $(B(L),\oplus,\ast,0)$ is a subalgebra of $L$, which is called the Boolean center of $L$ \cite{Chang1}.

\begin{proposition}\emph{\cite{Chang1} Let $L$ be an MV-algebra. Then the following statements are equivalent:}
\begin{enumerate}[(1)]
  \item \emph{$x\in B(L)$,
  \item $x\oplus y=x\vee y$,
  \item $x\odot y=x\wedge y$, for any $y\in L$.}
\end{enumerate}
\end{proposition}

\begin{proposition}\emph{\cite{Turunen} Let $L$ be an MV-algebra and $e\in B(L)$. Then the following properties hold:}
\begin{enumerate}[(1)]
  \item \emph{$e\wedge(x\odot y)=(e\wedge x)\odot(e\wedge y)$,
  \item $e\vee(x\odot y)=(e\vee x)\odot(e\vee y)$,
  \item $e\wedge(x\oplus y)=(e\wedge x)\oplus(e\wedge y)$,
  \item $e\vee(x\oplus y)=(e\vee x)\oplus(e\vee y)$,
  \item $e\odot(x\rightarrow y)=e\odot[(e\odot x)\rightarrow (e\odot y)]$, for any $x,y\in L$. }
\end{enumerate}
\end{proposition}

Let $L$ be an MV-algebra. A nonempty subset $I$ of $L$ is called an \emph{ideal} of $L$ if it satisfies: (1) $x,y\in I$ implies $x\oplus y\in I$; (2) $x\in I$, $y\in L$ and $y\leq x$ implies $y\in I$. An ideal $I$ of an MV-algebra $L$ is \emph{proper} if $I\neq L$. A proper ideal $I$ of an MV-algebra $L$ is called a prime ideal if for any $x,y\in L$ such that $x\wedge y\in I$, then $x\in I$ or $y\in I$ .A nonempty subset $I$ of an MV-algebra $L$ is called a \emph{lattice
ideal} of $L$ if it satisfies: (i) for all $x,y\in I$, $x\vee y\in I$; (ii) for all $x,y\in L$, if $x\in I$
and $y\leq x$, then $y\in I$. That is, a lattice ideal of an MV-algebra $L$ is the
notion of ideal in the underlying lattice $(L,\wedge,\vee)$. For any nonempty subset $H$ of $L$, the smallest lattice ideal containing $H$ is
called the lattice ideal generated by $H$. The lattice ideal generated by $H$ will be
denoted by $(H]$. In particular, if $H =\{t\}$, we write $(t]$ for $(\{t\}]$, $(t]$ is called a
principal lattice ideal of $L$. It is easy to check that $(t] =\downarrow t =\{x\in L|x\leq t\}$\cite {Mundici1,Mundici2,Gratzer}.

Let $I$ be an ideal of an MV-algebra $L$. We define a binary relation $\theta_I$ on $L$ as follows: for any $x,y\in L$, $(x,y)\in \theta_I$ if and only if $(x\odot y^\ast)\oplus(y\odot x^\ast)\in I$. Then, $\theta_I$ is a congruence relation on $L$. Thus, the binary relation $\leq$ on $L/I$ which is defined by $[x]\leq [y]$, if and only if $x^\ast\oplus y\in I$, is an order relation on $L/F$. For any $x\in L$, let $[x]_I$ be the equivalence class $[x]_{\theta_I}$ and $L/F=L/\theta_I=\{[x]_I|x\in L\}$. Then $L/I$ becomes an MV-algebra with the natural operations induced from those of $L$. For more details about ideals in MV-algebras \cite{Turunen}.

\begin{definition}\emph{\cite{Turunen} Given ordered sets $E,F$ and order-preserving mappings $f:E\longrightarrow F$ and $g:F\longrightarrow E$, we say that the pair $(f,g)$ establishes a \emph{Galois connection} between $E$ and $F$ if $fg\geq id_F$ and $gf\leq id_E$.}
\end{definition}

\section{On derivations of MV-algebras}

In this section, we investigate some derivations in an MV-algebra. Then we give some characterizations of additive derivations. Also, we discuss the relationship between additive derivations and ideals of MV-algebras..

\begin{definition}\emph{\cite{Alshehri} Let $L$ be an MV-algebra. A map $d:L\longrightarrow L$ is called a \emph{ derivation} on $L$ if it satisfies the following conditions: for any $x,y\in L$,}
\begin{center} $d(x\odot y)=(d(x)\odot y)\oplus (x\odot d(y))$.
\end{center}
\end{definition}

Now, we present some examples for derivations on MV-algebras.
\begin{example}\emph{Let $L$ be an MV-algebra. Define a map $d:L\longrightarrow L$ by $d(x)=0$ for all $x\in L$, then $d$ is a derivation on $L$, which is called a zero derivation}.
\end{example}
\begin{example}\emph{Let $L=\{0,a,b,c,d,1\}$ and operations $\oplus$ and $\ast$ be defined as follows}:\\
\begin{center}
\begin{tabular}{c|c c c c c c}
   $\oplus$ & $0$ & $a$ & $b$ & $c$ & $d$ & $1$\\
   \hline
   $0$ & $0$ & $a$ & $b$ & $c$ & $d$ & $1$\\
   $a$ & $a$ & $c$ & $d$ & $c$ & $1$ & $1$ \\
   $b$ & $b$ & $d$ & $b$ & $1$ & $d$ & $1$ \\
   $c$ & $c$ & $c$ & $1$ & $c$ & $1$ & $1$\\
   $d$ & $d$ & $1$ & $d$ & $1$ & $1$ & $1$ \\
   $1$ & $1$ & $1$ & $1$ & $1$ & $1$ & $1$
 \end{tabular} {\qquad}
\begin{tabular}{c|c c c c c c}
   $\ast$ & $0$ & $a$ & $b$ & $c$ & $d$ & $1$\\
   \hline
          & $1$ & $d$ & $c$ & $b$ & $a$ & $0$ \\
   \end{tabular}
\end{center}

\emph{Then $(\{0,a,b,c,d,1\},\oplus,\ast,0,1)$ is an MV-algebra. Define a map $d:L\longrightarrow L$ by $d(0)=d(a)=d(c)=0$, $d(b)=d(d)=d(1)=b$. One can check that $d$ is a derivation on $L$.}
\end{example}

\begin{proposition}\emph{\cite{Alshehri} Let $L$ be an MV-algebra and $d$ be a derivation on $L$. Then we have: for any $x,y\in L$,}
\begin{enumerate}[(1)]
  \item \emph{$d(0)=0$,
  \item $d(1)\in B(L)$,
  \item $d(x)\odot x^\ast=x\odot d(x^\ast)=0$,
  \item $d(x)\leq x$,
  \item $d(x)=d(x)\oplus (x\odot d(1))$.}
\end{enumerate}
\end{proposition}
\begin{definition}\emph{\cite{Alshehri} Let $L$ be an MV-algebra and $d$ be a derivation on $L$}.
\begin{enumerate}[(1)]
  \item \emph{$d$ is called an \emph{isotone derivation} provided that $x\leq y$ implies $d(x)\leq d(y)$ for all $x,y\in L$,
  \item $d$ is called an \emph{additive derivation} provided that $d(x\oplus y)=d(x)\oplus d(y)$ for all $x,y\in L$.}
\end{enumerate}
\end{definition}

\begin{example}\emph{Considering the MV-algebra $S_4=\{0,\frac{1}{3},\frac{2}{3},1\}$ in Example 2.2. Define a map $d:S_4\longrightarrow S_4$ by $d(0)=d(1)=d(\frac{1}{3})=0$, $d(\frac{2}{3})=\frac{1}{3}$. One can check that $d$ is a derivation on $S_4$, but it is not an additive derivation on $L$, since $d(\frac{1}{3}+\frac{2}{3})=d(1)=0\neq \frac{1}{3}=d(\frac{1}{3})+d(\frac{2}{3})$. Also, $d$ is not an isotone derivation on $L$, although $\frac{2}{3}\leq 1$, $d(\frac{2}{3})=\frac{1}{3}\geq 0=d(1)$. }
\end{example}

\begin{example}\emph{Considering the derivation $d$ in Example 3.3, one can see that $d$ is not only an additive and but also is an isotone derivation on $L$.}
\end{example}
\begin{proposition} \emph{Let $L$ be an MV-algebra and $d$ be an additive derivation on $L$. Then we have: for any $x,y\in L$,}
\begin{enumerate}[(1)]
  \item \emph{$d$ is an isotone derivation,
  \item $d(x)=d(1)\odot x$,
  \item $dd(x)=d(x)$,
  \item $d(x)\in B(L)$,
  \item $d(d(x)\rightarrow d(y))=d(x\rightarrow y)$,
  \item $Fix_d(L)=d(L)$, where $Fix_d(L)=\{x\in L|d(x)=x\}$,
  \item if $d(L)=L$, then $d=id_L$,
  \item Ker$(d)$ is an ideal of $L$, where Ker$(d)=\{x\in L|d(x)=0\}$.}
\end{enumerate}
\end{proposition}
\begin{proof} (1) If $x\leq y$, then $y=x\vee y=x\oplus(x^\ast\odot y)$. From Definition 3.5(2), we have $d(y)=d(x\oplus(x^\ast\odot y))=d(x)\oplus d(x^\ast\odot y)\geq d(x)$. Thus, $d(x)\leq d(y)$.

(2) From Proposition 3.4(2),(4) and (1), we have $d(x)=d(1)\wedge d(x)=d(1)\odot d(x)\leq d(1)\odot x$. On the other hand, we conclude from Definition 3.1 that $d(x)\geq x\odot d(1)$. Thus, $d(x)=d(1)\odot x$.

(3) From Proposition 3.4(2) and (2), we have $dd(x)=d(d(1)\odot x)=d(1)\odot d(1)\odot x=d(1)\odot x=d(x)$.

(4) From Definition 3.1 and (2), we have $d(x)=d(x\odot 1)=(d(x)\odot 1)\oplus (x\odot d(1))=d(x)\oplus (x\odot d(1))=d(x)\oplus d(x)$. Thus, $d(x)\in B(L)$.

(5) From Proposition 2.5(5), 3.4(2) and (2), we have $d(d(x)\rightarrow d(y))=d(1)\odot[(d(1)\odot x)\rightarrow(d(1)\odot y)]=d(1)\odot(x\rightarrow y)=d(x\rightarrow y)$. Thus, $d(d(x)\rightarrow d(y))=d(x\rightarrow y)$.

(6) Let $y\in d(L)$. So there exists $x\in L$ such that $y=d(x)$ and hence $d(y)=dd(x)=d(x)=y$, that is, $y\in Fix_d(L)$. Conversely, if $y\in Fix_d(L)$, then we have $y\in d(L)$. Therefore, $Fix_d(L)=d(L)$.

(7) For any $x\in L$, we have $x=d(x_0)$ for some $x_0\in L$. From (3), we have $d(x)=dd(x_0)=d(x_0)=x$. Therefore, $d=id_L$.

(8) From Proposition 3.4(1), we have $d(0)=0$, that is, $0\in$ Ker$(d)$. If $x,y\in$ Ker$(d)$, then $d(x)=d(y)=0$ and hence $d(x\oplus y)=d(x)\oplus d(y)=0$, that is, $x\oplus y\in$ Ker$(d)$. Finally, if $x\leq y$ and $y\in$ Ker$(d)$, then $d(x)\leq d(y)=0$ and hence $d(x)=0$, that is, $x\in$ Ker$(d)$.

\end{proof}
\begin{remark}\emph{ In \cite{Alshehri}, the results (1) and (8) in Proposition 3.8 are proved as Theorems 3.16, 3.17 under the additional condition that $L$ is an linearly ordered MV-algebra, but this condition is redundant as our proof shows.}
\end{remark}

\begin{theorem} \emph{Let $L$ be an MV-algebra and $d$ be a derivation on $L$. Then the following are equivalent:}
\begin{enumerate}[(1)]
  \item \emph{$d$ is an additive derivation,
  \item $d$ is an isotone derivation,
  \item $d(x)\leq d(1)$,
  \item $d(x)=d(1)\odot x$,
  \item $d(x\odot y)=d(x)\odot y=x\odot d(y)$,
  \item $d(x\wedge y)=d(x)\wedge d(y)$,
  \item $d(x\vee y)=d(x)\vee d(y)$,
  \item $d(x\odot y)=d(x)\odot d(y)$,
  \item $d(x)\leq y$ if and only if $d(x)\leq d(y)$,
  \item $d(x)\rightarrow d(y)=d(x)\rightarrow y$.}
\end{enumerate}
\end{theorem}
\begin{proof} $(1)\Rightarrow (2)$ It follows from Proposition 3.8(1)

$(2)\Rightarrow (3)$ It is straightforward.

$(3)\Rightarrow (4)$ It follows from Proposition 3.8(2).

$(4)\Rightarrow (1)$ From Proposition 2.5(4) and Proposition 3.8(2), we have $d(x\oplus y)=d(1)\odot(x\oplus y)=(d(1)\odot x)\oplus (d(1)\odot y)=d(x)\oplus d(y)$, that is, $d(x\oplus y)=d(x)\oplus d(y)$.

From the above proof, one can see the statements (1)-(4) are equivalent.

$(4)\Rightarrow (5)$ Let $d(x)=d(1)\odot x$ for all $x\in L$. It follows that $d(x\odot y)=d(1)\odot x\odot y=x\odot (d(1)\odot y)$, that is, $d(x\odot y)=d(x)\odot y=x\odot d(y)$ for any $x,y\in L$.

$(5)\Rightarrow (4)$ Taking $y=1$ in (5), we have $d(x)=d(1)\odot x$ for any $x\in L$.

$(4)\Rightarrow (6)$ Let $d(x)=d(1)\odot x$ for any $x\in L$. It follows that $d(x\wedge y)=d(1)\odot(x\wedge y)=d(1)\wedge (x\wedge y)=d(1)\wedge d(1)\wedge x\wedge y=d(x)\wedge d(y)$ for any $x,y\in L$.

$(4)\Rightarrow (7)$ Since $(L,\wedge,\vee,0,1)$ is a distributive lattice. From (4), we have $d(x\vee y)=d(1)\odot(x\vee y)=d(1)\wedge(x\vee y)=(d(1)\wedge x)\vee(d(1)\wedge y)=(d(1)\odot x)\vee(d(1)\odot y)=d(x)\vee d(y)$ for any $x,y\in L$.

$(7)\Rightarrow (2)$ Let $x\leq y$, we have $x\vee y=y$. From (7), we $d(y)=d(x\vee y)=d(x)\vee d(y)\geq d(x)$ for any $x,y\in L$. Therefore, $d$ is an isotone derivation.

$(4)\Rightarrow (8)$ From (4), we have $d(x\odot y)=d(1)\odot (x\odot y)=d(1)\odot d(1)\odot x\odot y=d(x)\odot d(y)$ for any $x,y\in L$.

$(6)\Rightarrow(3)$, $(7)\Rightarrow (3)$, $(8)\Rightarrow (3)$ are straightforward.

$(2)\Rightarrow (9)$  For all $x,y\in L$, assume that $d(x)\leq y$, we have $dd(x)\leq d(y)$. By 3.8(2), we get $dd(x)=d(x)$. Thus $d(x)\leq d(y)$. Conversely, suppose that $d(x)\leq d(y)$, we get $d(x)\leq d(y)\leq y$ for all $x,y\in L$.

$(9)\Rightarrow (2)$ Let $x\leq y$, we have $d(x)\leq x\leq y$ and hence $d(x)\leq d(y)$ for all $x,y\in L$.

$(2)\Rightarrow (10)$ Suppose that $d$ is an isotone derivation on $L$. From $d(y)\leq y$ for any $y\in L$, it follows that $d(x)\rightarrow d(y)\leq d(x)\rightarrow y$. On the other hand, let $t\leq d(x)\rightarrow y$ for all $t\in L$, we can obtain $d(x)\odot t\leq y$. Since $d$ is an isotone derivation, we have $d(d(x)\odot t)\leq d(y)$ for all $x,y,t\in L$. From $d(x\odot y)=(d(x)\odot y)\oplus(x\odot d(y))$, we get $d(x)\odot y\leq d(x\odot y)$ for all $x,y\in L$. It follows $d(d(x))\odot t\leq d(d(x)\odot t)$. By Proposition 3.8(3), we have $d(x)\odot t\leq d(d(x)\odot t)\leq d(y)$. Hence $t\leq d(x)\rightarrow d(y)$ for all $t\in L$, which implies $d(x)\rightarrow y\leq d(x)\rightarrow d(y)$ for all $x,y\in L$. Therefore, we obtain $d(x)\rightarrow d(y)=d(x)\rightarrow y$ for any $x,y\in L$.

$(10)\Rightarrow (2)$ Assume that $d(x)\rightarrow d(y)=d(x)\rightarrow y$ for all $x,y\in L$. For any $x,y\in L$,let $x\leq y$, by Proposition 3.4(4), we have $d(x)\odot 1=d(x)\leq x\leq y$. It follows that $1\leq d(x)\rightarrow y=d(x)\rightarrow d(y)$, which implies $d(x)\leq d(y)$.
\end{proof}

\begin{remark}\begin{enumerate}[(1)]
                \item \emph{From the above theorem, one can see that isotone derivations are equivalent to additive derivations on MV-algebras.
                \item From the Example 3.3, one can check that an additive derivation is not a homomorphism on an MV-algebra $L$ in general, since $d(a^\ast)=b\neq 1=(d(a))^\ast$.
                \item From the Example 3.3, one can check that the fixed point set of an additive derivation $d$ is not a subalgebra of an MV-algebra $L$ in general, since $0^\ast=1\notin \{0,b\}=Fix_d(L)$.
                \item From (4) of the above theorem, one can see that every additive derivation $d$ on an MV-algebra $L$ is completely defined by the image $d(1)$ of the $1$.
                \item It is easily seen that every additive derivation is a lattice derivation in the sense of Ferrari \cite{Ferrari}, which is a unary map $d:L\rightarrow L$ satisfying conditions Theorem 3.10(7) and $d(x\wedge y)=(d(x)\wedge y)\vee (x\wedge d(y))$. However, the converse is not true in general. Moreover, someone proved that a lattice derivation on an MV-algebra is an additive derivation if and only if $d(L)\subseteq B(L)$ in \cite{Ghorbain}.}
              \end{enumerate}
\end{remark}

The following theorem shows that the fixed point set $Fix_d(L)$ of additive derivations in an MV-algebra $L$ has the same structure as $L$, which reveals the essence of the fixed point set of  additive derivations.

\begin{theorem}\emph{Let $L$ be an MV-algebra and $d$ be an additive derivation on $L$. Then $(Fix_d(L), \oplus, \neg, 0)$ is an MV-algebra, where $\neg x=d(x^\ast)=d(1)\odot x^\ast=(x\oplus (d(1))^\ast)^\ast$ for any $x\in Fix_d(L)$.}
\end{theorem}

\begin{proof} First, we will prove that $(Fix_d(L),\oplus,0)$ is a commutative monoid. From Definition 3.5(2), we have $Fix_d(L)$ is closed under $\oplus$. It follows that $(Fix_d(L),\oplus)$ is commutative semigroup. For all $x\in Fix_d(L)$, we can obtain that $x\oplus 0=x$, that is, $0$ is a unital element.

Next, we will prove that $Fix_d(L)$ is closed under $\neg$. For all $x\in Fix_d(L)$, from Proposition 3.4(2) and 3.8(2), we have $d(\neg x)=d(d(x^\ast))=d(1)\odot d(1)\odot x^\ast=d(1)\odot x^\ast=d(x^\ast)=\neg x=\neg d(x)$, that is, $Fix_d(L)$ is closed under $\neg$.

Finally, we will verify the remaining axioms of an MV-algebra.

(MV2) For all $x\in Fix_d(L)$, from Theorem 3.10(3), we have $\neg\neg x=(\neg x\oplus (d(1)^\ast)^\ast=((x\oplus (d(1))^\ast)\oplus (d(1))^\ast)^\ast=(x\oplus (d(1))^\ast)\odot d(1)=x\wedge d(1)=d(x)\wedge d(1)=d(x)=x$, that is, $\neg\neg x=x$.

(MV3) For all $x\in Fix_d(L)$, from Theorem 3.10(3), we have $x\oplus \neg 0=x\oplus d(1)=d(x)\oplus d(1)=d(x)\vee d(1)=d(1)$, that is, $x\oplus \neg 0=d(1)$.

(MV4) For any $x,y\in Fix_d(L)$, we have $\neg(\neg x\oplus y)\oplus y=(\neg x\oplus y\oplus (d(1))^\ast)^\ast\oplus y=(x\oplus (d(1))^\ast)^\ast\oplus y)\oplus(d(1))^\ast)^\ast\oplus y=((x\oplus (d(1))^\ast)\odot d(1)\odot y^\ast)\oplus y=((x\wedge d(1))\odot y^\ast)\oplus y=x\vee y$. in the similar way, one can prove that $\neg(x\oplus \neg y)\oplus x=x\vee y$, and hence $\neg(\neg x\oplus y)\oplus y=\neg(x\oplus \neg y)\oplus x$.

Therefore, we obtain that $(Fix_d(L), \oplus, \neg, 0)$ is an MV-algebra.
\end{proof}

\begin{theorem}\emph{Let $L$ be an MV-algebra and $d$ an additive derivation on $L$. Then we have the following properties:}
\begin{enumerate}[(1)]
  \item \emph{$d:L\longrightarrow Fix_d(L)$ is a surjective homomorphism,
  \item $\bar{d}:L/Ker(d)\longrightarrow Fix_d(L)$ is an isomorphism.}
\end{enumerate}
\end{theorem}
\begin{proof}$(1)$ It follows from Theorem 3.10 and 3.12.

$(2)$ From Proposition 3.8(8), we obtain that $Ker(d)$ is an ideal of $L$. Let $x\sim_{Ker(d)}y$. Then $(x\ominus y)\oplus(y\ominus x)\in Ker(d)$, which implies $d((x\ominus y)\oplus(y\ominus x))=d(x\ominus y)\oplus d(y\ominus x)=0$. Furthermore, from Theorem 3.10(8), we have $d(x\ominus y)=d(x)\odot d(y^\ast)=d(x)\odot y^\ast=d(y\ominus x)=d(y)\odot x^\ast=0$, that is, $d(x)\leq y$ and $x\leq d(y)$. Further by Theorem 3.10(9), we have $d(x)=d(y)$. Thus, $\bar{d}$ is well defined. Moreovwe, it follows from $(1)$ that $\bar{d}:L/Ker(d)\longrightarrow Fix_d(L)$ is an isomorphism.
\end{proof}

\begin{theorem}\emph{Let $L$ be an MV-algebra and $d:L\rightarrow L$ be a map on $L$ such that $d(L)\subseteq B(L)$. Then the following statements are equivalent:}
\begin{enumerate}[(1)]
  \item \emph{$d$ is an additive derivation on $L$,
  \item $d(x)=d(1)\odot x$,
  \item $d(x\odot y)=d(x)\odot y=x\odot d(y)$.}
\end{enumerate}
\end{theorem}

\begin{proof} $(1)\Rightarrow (2)$  It follows from Proposition 3.8(2).

$(2)\Rightarrow (3)$ The proof of that is similar to Theorem 3.10 $(4)\Rightarrow (5)$.

$(3)\Rightarrow (1)$ Taking $y=1$ in (3), we have $d(x)=d(1)\odot x$. From $d(L)\subseteq B(L)$, we have $d(x\odot y)=d(x\odot y)\oplus d(x\odot y)=(d(x)\odot y)\oplus (x\odot d(y))$, that is, $d$ is a derivation on $L$. Moreover, from Theorem 3.10 $(1)\Leftrightarrow (4)$, we obtain that $d$ is an additive derivation on $L$.
\end{proof}

\begin{corollary}\emph{Let $L$ be a Boolean algebra and $d:L\rightarrow L$ be a map on $L$. Then the following statements are equivalent:}
\begin{enumerate}[(1)]
  \item \emph{$d$ is an additive derivation on $L$,
  \item $d(x)=d(1)\wedge x$,
  \item $d(x\wedge y)=d(x)\wedge y=x\wedge d(y)$.}
\end{enumerate}
\end{corollary}

\begin{proof}It follows from Theorem 3.14.
\end{proof}

\begin{theorem}\emph{Let $L$ be an MV-algebra and $d$ be an additive derivation on $L$. Then the following statements are equivalent:}
\begin{enumerate}[(1)]
  \item \emph{$d$ is an identity map,
  \item $d(x^\ast)=(d(x))^\ast$,
  \item $d(1)=1$,
  \item $d$ is homorphism,
  \item $d$ is surjection,
  \item $d$ is one to one,
  \item $d$ is isomorphic.}
\end{enumerate}
\end{theorem}

\begin{proof} $(1)\Rightarrow (2)$ It is straightforward.

$(2)\Rightarrow (3)$ From Proposition 3.4(1), we have $d(1)=d(0^\ast)=(d(0))^\ast=0^\ast=1$.

$(3)\Rightarrow (1)$ From Proposition 3.8(2), we have $d(x)=d(1)\odot x=1\odot x=x$.

$(2)\Rightarrow (4)$ From Proposition 3.4(1) and Theorem 3.10(8), we obtain that $d$ is homorphism.

$(4)\Rightarrow (2)$ It is straightforward.

$(1)\Rightarrow (5)$ It is straightforward.

$(5)\Rightarrow (1)$ Suppose that $d$ is surjection, then there exists $x\in L$ such that $d(x)=1$. From Proposition 3.4(4), we have $1=d(x)\leq x$ and hence $x=1$, that is, $d(1)=1$. Further from $(3)\Rightarrow (1)$, we get that $d$ is an identity map.

$(1)\Rightarrow (6)$ It is straightforward.

$(6)\Rightarrow (1)$ Suppose that $d$ is one to one, by Proposition 3.8(2), we have $d((d(1))^\ast)=d(1)\odot (d(1))^\ast=0$ and hence $d(1)=1$.  hence $x=1$, that is, $d(1)=1$. From $(3)\Rightarrow (1)$, we get that $d$ is an identity map. From $(3)\Rightarrow (1)$, we get that $d$ is an identity map.

$(1)\Leftrightarrow (7)$ It follows from $(1)\Leftrightarrow (4)$, $(1)\Leftrightarrow (5)$, $(1)\Leftrightarrow (6)$.

\end{proof}

In what follows, we discuss the relationship between ideals and derivations on MV-algebras.\\

The following example shows that the fixed point set Fix$_d(L)$ of a derivation $d$ is not an ideal of an MV-algebra $L$.

\begin{example}\emph{Considering $S_3=\{0,\frac{1}{2},1\}$ in Example 2.2 and defining a map $d:S_3\longrightarrow S_3$ by $d(0)=d(1)=0$, $d(\frac{1}{2})=\frac{1}{2}$. One can check that $d$ is a derivation, while it is not an additive derivation, and Fix$_d(L)=\{0,\frac{1}{2}\}$ is not an ideal of $L$ since $\frac{1}{2}\oplus \frac{1}{2}=1\notin$ Fix$_d(L)$.}
\end{example}

However, if the conditions are strengthened, we can obtain the following result.

\begin{proposition}\emph{Let $(L,\oplus,\ast,0,1)$ be an MV-algebra and $d$ be an additive derivation on $L$. Then Fix$_d(L)$ is an ideal of $L$.}
\end{proposition}
\begin{proof} First, from Proposition 3.4(1), we have $d(0)=0$ and hence Fix$_d(L)$ is not a non-empty subset of $L$. Next, we will prove that Fix$_d(L)$ is a downset of $L$. Let $x\leq y$ and $y\in Fix_d(L)$. Then we have $d(x)=d(x\wedge y)=d((x\oplus y^\ast)\odot y)=(d(x\oplus y^\ast)\odot y)\oplus((x\oplus y^\ast)\odot d(y))=(d(x\oplus y^\ast)\oplus x$ and hence $x\leq d(x)\leq x$. Thus, $d(x)=x$, that is, Fix$_d(L)$ is a downset of $L$. Finally, from the definition of additive derivation, we have Fix$_d(L)$ is closed under the operation $\oplus$. Therefore, Fix$_d(L)$ is an ideal of $L$.
\end{proof}

The following is a counterexample showing that the converse of Proposition 3.18 may not hold.

\begin{example}\emph{Considering the Example 3.6, one can check that Fix$_d(L)=\{0\}$ is an ideal of $L$. However, $d$ is not an additive derivation on $L$.}
\end{example}

Inspired by Proposition 3.18, it is natural to ask that whether there exists an additive derivation $d$ such that Fix$_d(L)=I$ for given ideal $I$ in an MV-algebra $L$.\\

For the above question, we give the positive answer under certain conditions.

%\begin{proposition}\emph{Let $(L,\oplus,\ast,0,1)$ be an MV-algebra and $I$ be a finite ideal of $L$ such that $I\subseteq B(L)$. Then there exists an additive derivation $d$ such that Fix$_d(L)=I$.}
%\end{proposition}
%\begin{proof} Since $I$ is a finite ideal of $L$, then $t=\bigvee_{a\in I}$. Define a map $d:L\longrightarrow L$ by $d(x)=t\odot x$, for all $x\in L$. Moreover, since $t\in I$ and $d(x)\leq t$, for all $x\in L$, then we get that $d(L)\subseteq I\subseteq B(L)$. Thus, by Theorem 3.13, $d$ is an additive derivation on $L$. Now, we will prove that Fix$_d(L)=I$. If $x\in I$, then $d(x)=x\odot t=x\wedge t=x$ and hence $x\in Fix_d(L)$. Further by $t\in B(L)$, we have Fix$_d(L)=d(L)$. Thus, we have $Fix_d(L)\subseteq I$ and hence Fix$_d(L)=I$.
%\end{proof}

\begin{proposition} \emph{Let $L$ be a Boolean algebra and $I$ be a non-void prime ideal of $L$. Then there exists an additive derivation $d$ such that Fix$_d(L)=I$. Indeed, a map $d:L\longrightarrow L$ that is defined by: for all $t,x\in L$,}
\begin{center}
$d(x)=
\begin{cases}
x, & x\in I \\
x\wedge t, & x\in L/I,
\end{cases}$
\end{center}

\emph{is an additive derivation satisfying Fix$_d(L)=I$.}
\end{proposition}

In fact, the above proposition shows under some suitable conditions, there exists an additive derivation on an MV-algebra $L$ such that Fix$_d(L)=I$ for given ideal $I$ of $L$. Furthermore, we have the following {\bf open problem:}
\begin{enumerate}[(1)]
 \item For any ideal $I$, which is not an lattice ideal, of a general MV-algebra $L$, whether there exists an additive derivation $d$ such that Fix$_d(L)=I$.
\end{enumerate}

\section{Boolean derivation and their applications}

In this sections, we investigate Boolean additive derivations and their adjoint  derivations. In particular, we prove that  every MV-algebras are isomorphic to the direct product of  the fixed point set of Boolean additive derivations and that of their adjoint derivations. Finally, we show that the structure of a Boolean algebra is completely determined by its set of all Boolean additive (implicative) derivations.\\

In what follows, let $L$ be an MV-algebra and $a\in L$, we define four maps as follows:

\begin{enumerate}[(1)]
  \item $d_a:L\rightarrow L$ such that $d_a(x)=x\odot a$, for all $x\in L$,
  \item $d_{a^*}:L\rightarrow L$ such that $d_{a^*}(x)=x\ominus a$, for all $x\in L$,
  \item $g_a:L\rightarrow L$ such that $g_a(x)=a\oplus x$, for all $x\in L$,
  \item $g_{a^*}:L\rightarrow L$ such that $g_{a^*}(x)=a\rightarrow x$, for all $x\in L$.
\end{enumerate}

\begin{theorem}\emph{Let $L$ be an MV-algebra. Then the following statements are equivalent:}
\begin{enumerate}[(1)]
  \item \emph{$L$ is a Boolean algebra,
  \item $d_a$ is an additive derivation on $L$, for all $a\in L$,
  \item $d_{a^*}$ is an additive derivation on $L$, for all $a\in L$.}
\end{enumerate}
\end{theorem}

\begin{proof}$(1)\Rightarrow (2)$ If $L$ is a Boolean algebra, then we get that $x\oplus y=x\vee y$ for any $x,y\in L$. Further by Corollary 3.15, we have $d_a$ is an additive derivation on $L$.

$(2)\Rightarrow (1)$ Assume that $d_a$ is an additive derivation on $L$. From Definition 3.1, we have $a=d_a(1)=d_a(1\odot 1)=(d_a(1)\odot 1)\oplus (1\odot d_a(1))=a\oplus a$ for any $a\in L$ and hence $L\subseteq B(L)$. Therefore, $L$ is a Boolean algebra.

$(1)\Leftrightarrow (3)$ The proof of $(1)\Leftrightarrow (3)$ is similar to that of $(1)\Leftrightarrow (2)$.

\end{proof}

\begin{remark} \emph{ From Theorem 4.1, we obtain that $d_a$ is an additive derivation on an MV-algebra if and only if $a\in B(L)$. Based on the above consideration, we called $d_a$ a Boolean additive  derivation on an MV-algebra $L$.}
\end{remark}

Next, we will discuss the adjoint derivation of Boolean additive  derivations. First, we introduce new derivation on MV-algebras.

\begin{definition}\emph{Let $L$ be an MV-algebra. A map $g:L\rightarrow L$ is called an \emph{implication derivation} on $L$ if it preserves $\rightarrow$ and satisfies the following conditions: for any $x,y\in L$},
\begin{center}  $g(x\rightarrow y)=(g(x)\rightarrow y)\oplus (x\rightarrow g(y))$.
\end{center}
\end{definition}

\begin{example}\emph{Let $L=\{0,a,b,1\}$ be  a chain and operations $\oplus$ and $\ast$ be defined as follows}:\\
\begin{center}
\begin{tabular}{c|c c c c }
   $\oplus$ & $0$ & $a$ & $b$ & $1$\\
   \hline
   $0$ & $0$ & $a$ & $b$ &  $1$\\
   $a$ & $a$ & $a$ & $1$ &  $1$ \\
   $b$ & $b$ & $1$ & $b$ &  $1$ \\
   $1$ & $1$ & $1$ & $1$ &  $1$
 \end{tabular} {\qquad}
\begin{tabular}{c|c c c c}
   $\ast$ & $0$ & $a$ & $b$ &  $1$\\
   \hline
          & $1$ & $b$ & $a$ &  $0$ \\
   \end{tabular}
\end{center}
\emph{Then $(\{0,a,b,1\},\oplus,\ast,0)$ is an MV-algebra. Define a map $g:L\longrightarrow L$ by $g(0)=g(a)=a$, $g(b)=g(1)=1$. One can check that $g$ is an implicative derivation on $L$, while $g$ is not a homorphism on $L$, since $g(0)\neq 0$. }
\end{example}

\begin{theorem}\emph{Let $L$ be an MV-algebra. Then the following statements are equivalent:}
\begin{enumerate}[(1)]
  \item \emph{$L$ is a Boolean algebra,
  \item $g_{a^*}$ is an implicative derivation on $L$, for all $a\in L$,
  \item $g_a$ is an implicative derivation on $L$, for all $a\in L$ .}
\end{enumerate}
\end{theorem}

\begin{proof} $(1)\Rightarrow (2)$ Let $L$ be a Boolean algebra and $a\in L$. First, we will prove that $g_{a^*}(x\rightarrow y)=g_{a^*}(x)\rightarrow g_{a^*}(y)$. In particular, since $(a\odot
 x\rightarrow y)\odot a\odot(a\rightarrow x)\leq (a\odot x\rightarrow y)\odot
 a\odot x\leq y$, we have $a\odot x\rightarrow y\leq a\odot(a\rightarrow x)\rightarrow y$ and hence
 $g_{a^\ast}(x\rightarrow y)=a\rightarrow(x\rightarrow y)=x\rightarrow(a\rightarrow y)=a\odot x\rightarrow y\leq
 a\odot(a\rightarrow x)\rightarrow y=(a\rightarrow x)\rightarrow(a\rightarrow y)=g_{a^\ast}(x)\rightarrow g_{a^\ast}(y)$. On the other hand, since
 $x\leq a\rightarrow x$, it easy to show that $a\odot x\leq a\odot(a\rightarrow x)$ and hence that $a\odot(a\rightarrow x)\rightarrow y\leq a
 \odot x\rightarrow y$, that is, $g_{a^*}(x)\rightarrow g_{a^*}(y)\leq g_{a^*}(x\rightarrow y)$. Thus, we have
 $g_{a^*}(x\rightarrow y)=g_{a^*}(x)\rightarrow g_{a^*}(y)$.

 Next, from Proposition 2.5, we have $g_{a^*}(x\rightarrow y)=a\oplus (x\rightarrow y)=x^\ast\oplus g_{a^*}(y)=((g_{a^*}(x))\vee x^\ast)\oplus (g_a(y)\vee y)=((g_a(x))^\ast\oplus y)\oplus (x^\ast\oplus g_a(y))=(g_{a^*}(x)\rightarrow y)\oplus(x\rightarrow g_{a^*}(y)$.

 Combing them, one can see that $g_{a^*}$ is an implicative derivation on $L$.

 $(2)\Rightarrow (1) $ Assume that $g_{a^*}$ is an implicative derivation on $L$, for any $a\in L$. Then $g_{a^*}(x\rightarrow y)=g_{a^*}(x)\rightarrow g_{a^*}(y)$ holds for all $x,y\in L$,
 then taking $x=a$ and $y=a\odot a$ we have $a\rightarrow(a\rightarrow a\odot a)=(a\rightarrow a)\rightarrow(a\rightarrow
 a\odot a)$. Since $a\rightarrow(a\rightarrow a\odot a)=a\odot a\rightarrow a\odot a=1$, we have
 $1=1\rightarrow (a\rightarrow a\odot a)=a\rightarrow a\odot a$. This implies that $a\leq a\odot a$ and hence that $a=a\odot a$, that is $L$ is a Boolean algebra.

 $(1)\Leftrightarrow (3)$ The proof of $(1)\Leftrightarrow (3)$ is similar to that of $(1)\Leftrightarrow (2)$.
\end{proof}

\begin{remark} \emph{ From Theorem 4.5, we obtain that $g_a$ is an implication derivation on an MV-algebra if and only if $a\in B(L)$. Based on the above consideration, we called $g_a$ a Boolean implication  derivation on an MV-algebra $L$.}
\end{remark}

\begin{definition}\emph{ Let $L$ be an MV-algebra and $\mu_a$ be a Boolean additive derivation on $L$. The Boolean additive derivation $\mu_a$ is called residuated if there exists a Boolean  implicative derivation $\nu_a$ on $L$ such that a pair $(\mu_a,\nu_a)$ forms a Galois connection.}
\end{definition}

The Boolean implicative derivation $\nu_a$ is called the adjoint derivation of the Boolean additive derivation $\mu_a$. \\

From the notion of Galois connections, we have that if the Boolean additive derivation $\mu_a$ is residuated, then $\mu_a$ and it's  adjoint derivation must be isotone. In particular, if the Boolean additive  derivation $\mu_a$ has the adjoint derivation $\nu_a$, then the adjoint of $\mu_a$ unique. Therefore, we shall denote this unique $\nu_a$ by $\mu^{\ast}_a$.

\begin{example}\emph{Consider the MV-algebra in Example 4.7. Define two maps $\nu_b(0)=\nu_b(b)=b$, $\nu_b(a)=\nu_b(1)=1$, $\mu_b(0)=\mu_b(a)=0$, $\mu_b(b) =\mu_b(1)=b$, respectively. One can check that $\nu_b$ and $\mu_b$ are Boolean implicative derivation and Boolean additive derivation on $L$, respectively. Further, one can check that $(\nu_b,\mu_b)$ froms a Galois connection.}
\end{example}

\begin{theorem}\emph{ Let $L$ be an MV-algebra and $a\in B(L)$. The the map $g_a(g_a^{\ast})$ is the adjoint derivation of the Boolean additive derivation $d^{\ast}_a(d^{\ast}_{a^{\ast}})$ on $L$, that is, $d^{\ast}_a=g_{a^{\ast}}(d^{\ast}_{a^{\ast}}=g_{a})$.}
\end{theorem}
\begin{proof} First, from Theorems 4.1 and 4.5, we get $d_a$ and $g_{a^{\ast}}$ are Boolean additive derivation and Boolean implicative derivation on $L$, respectively. Moreover, for all $x,y\in L$, let $x\leq y$, we conclude from Proposition 2.3 that $g_{a^{\ast}}(x)=a\rightarrow x\leq a\rightarrow y=g_{a^{\ast}}(y)$, that is,  $g_{a^{\ast}}$ is isotone. For all $x,y\in L$, we have $d_a(x)=a\odot x\leq y$ if and only if $x\leq a\rightarrow y=g_{a^{\ast}}(y)$. Thus, $(d_a,g_{a^{\ast}})$ forms a Galois connection.
 Combining them, we obtain that $g_{a^{\ast}}$ is the adjoint derivation of $d_a$, that is, $d^{\ast}_a=g_{a^{\ast}}$. In the similar way, one can prove that $d^{\ast}_{a^{\ast}}=g_{a}$.
\end{proof}

In what follows, we denote by $Fix_{d_a}(L)$ the set of all fixed points for $d_a$ and $Fix_{g_a}(L)$ the set of all fixed points of $L$ for $g_a$, respectively.

\begin{theorem}\emph{ Let $L$ be an MV-algebra and $a\in B(L)$. Then $(Fix_{d_a}(L),\oplus,\neg_1,0,a)$ is an MV-algebra, where $\neg_1 x=d_a(x^\ast)$, for all $x\in L$.}
\end{theorem}
\begin{proof}The proof is similar to that of Theorem 3.12.
\end{proof}

\begin{theorem}\emph{ Let $L$ be an MV-algebra and $a\in B(L)$. Then $(Fix_{g_{a}}(L),\oplus,\circ_1,a,1)$ is an MV-algebra, where $ x^{\circ_1}=g_{a}(x^\ast)$, for all $x\in L$.}
\end{theorem}
\begin{proof} First, we prove that $(Fix_{g_{a}}(L),\oplus,a)$ is a commutative monoid. One can easy to check that $Fix_{g_{a}}(L)$ is closed under $\oplus$. It follows that $(Fix_{g_{a}}(L),\oplus)$ is commutative semigroup. For all $x\in Fix_{g_{a}}(L)$, we can obtain that $x\oplus a=x$, that is, $a$ is a unital element.

Next, we will prove that $(Fix_{g_{a}}(L)$ is closed under $\circ_1$. For all $x\in Fix_{g_{a}}(L)$, we have $g(\neg x)=g(g(x^\ast))=a\oplus a\oplus x^\ast=a\oplus x^\ast=g(x^\ast)=x^{\circ_1}=(g(x))^{\circ_1}$, that is, $Fix_{g_{a}}(L)$ is closed under $\circ_1$.

Finally, we will verify the remaining axioms of an MV-algebra.

(MV2) For all $x\in Fix_{g_{a}}(L)$, we have $ {x^{\circ_1}}^{\circ_1}=( x^{\circ_1}\odot (a^\ast)^\ast=((x\odot a)^\ast)\oplus a^\ast)^\ast=(x\odot a^\ast)\oplus a=x\vee a=d(x)\vee d(a)=d(x)=x$, that is, ${x^{\circ_1}}^{\circ_1}=x$.

(MV3) For all $x\in Fix_{g_{a}}(L)$, we have $x\oplus a^{\circ_1}=x\oplus 1=1$.

(MV4) For any $x,y\in Fix_{g_{a}}(L)$, we have $(x^{\circ_1}\oplus y)^{\circ_1}\oplus y=( x^{\circ_1}\odot y\odot a^\ast)^\ast\odot y=(x\odot a^\ast)^\ast\odot y)\odot a^\ast)^\ast\odot y=((x\odot a^\ast)\oplus a\oplus y^\ast)\odot y=((x\vee a)\oplus y^\ast)\odot y=x\wedge y$. In the similar way, one can prove that $(x^{\circ_1}\odot y^{\circ_1})^{\circ_1}\odot x=x\wedge y$, and hence $(x^{\circ_1}\oplus y)^{\circ_1}\oplus y=(x\oplus y^{\circ_1})^{\circ_1}\oplus x$.

Therefore, we obtain that $(Fix_{g_{a}}(L), \oplus, \circ_1, a,1)$ is an MV-algebra.
\end{proof}

\begin{corollary}\emph{ Let $L$ be an MV-algebra and $a\in B(L)$. Then $(Fix_{d_{a^{\ast}}}(L),\oplus,\neg_2,0,a^\ast)$ is an MV-algebra, where $\neg_2 x=d_{a^{\ast}}(x^\ast)$, for all $x\in L$.}
\end{corollary}
\begin{proof} The proof is similar to that of Theorem 4.10.
\end{proof}

\begin{corollary}\emph{ Let $L$ be an MV-algebra and $a\in B(L)$. Then $(Fix_{g_{a^{\ast}}}(L),\oplus,\circ_2,a^\ast,1)$ is an MV-algebra, where $ x^{\circ_2}=g_{a^{\ast}}(x^\ast)$, for all $x\in L$.}
\end{corollary}

\begin{proof} The proof is similar to that of Theorem 4.11.
\end{proof}

The following theorem shows the relationship between $Fix_{d_a}(L)$ and $Fix_{g_{a^{\ast}}}(L)$.

\begin{theorem}\emph{ Let $L$ be an MV-algebra and $a\in B(L)$. Then MV-algebras  $(Fix_{d_a}(L),\oplus,\neg_1,0,a)$ and
$(Fix_{g_{a^{\ast}}}(L),\oplus,\circ_2,a^\ast,1)$ are isomorphic.}
\end{theorem}
\begin{proof} For all $a\in L$, let $f:Fix_{d_a}(L)\longrightarrow Fix_{g_{a^{\ast}}}(L)$ be defined by $f(x)=a^\ast \oplus x$ for all $x\in Fix_{d_a}(L)$. Clearly, $f$ is a map from $Fix_{d_a}(L)$ to $Fix_{g_{a^{\ast}}}(L)$, that is, $f$ is well defined.

$(1)$ For all $x,y\in Fix_{d_a}(L)$, we have $f(x\oplus y)=a^\ast\oplus(x\oplus y)=a^\ast \oplus a^\ast \oplus x\oplus y=(a^\ast \oplus x)\oplus (a^\ast \oplus y)=f(x)\oplus f(y)$, that is, $f(x\oplus y)=f(x)\oplus f(y)$.

$(2)$ For all $x\in Fix_{d_a}(L)$, we have $f(\neg_1 x)=a^\ast\oplus (a\odot x^\ast)=a^\ast\vee(a\wedge x^\ast)=(a^\ast \vee a)\wedge(a^\ast\vee x)=a^\ast\oplus x$. One the other hand, we have $ (f(x))^{\circ_2}=(a^\ast \oplus x)^{\circ_2}=a^\ast \oplus (a^\ast \oplus x)=a^\ast\oplus x$ and hence $f(\neg_1 x)=(f(x))^{\circ_2}$.

$(3)$ For all $x,y\in Fix_{d_a}(L)$, that is, $x=a\odot x$ and $y=a\odot y$, if $f(x)=f(y)$, then $a^\ast\oplus x=a^\ast\oplus y$. From $a\odot x\leq x$, we have $x\leq a\rightarrow x=a^\ast \oplus x$, then $x\leq a^\ast \oplus y=a\rightarrow y$. It follows that $a\odot x\leq y$, which implies $x\leq y$. Similarly, we can prove $y\leq x$ and hence $x=y$. Consequently, we obtain that $f$ is injective.

$(4)$  One can easily prove the fact that $x^\ast \vee y=x^\ast\oplus(x\odot y)=y$ if and only if there exists $z\in L$ such that $x^\ast\oplus z=y$ for all $x,y\in L$. From the above fact, we will prove that $f$ is surjective. Now, for all $x\in Fix_{g_{a^{\ast}}}(L)$, then $x= g_{a^{\ast}}(x)=a^\ast\oplus x$. Using the above fact, we have that $f(a\odot x)=a^\ast\oplus(a\odot x)=a^\ast\oplus(a\odot(a^\ast\oplus x))=a^\ast\oplus x=x$. Thus, we conclude that $f$ is surjective. Also, it is easy to verify that $f^{-1}(x)=a\odot x$ is a homomorphism from  $Fix_{g_{a^{\ast}}}(L)$ to $Fix_{d_a}(L)$.

Combining them, we obtain that $f$ is an isomorphism from $(Fix_{d_a}(L),\oplus,\neg_1,0,a)$ to
$(Fix_{g_{a^{\ast}}}(L),\oplus,\circ_2,a^\ast,1)$. Therefore, MV-algebras $(Fix_{d_a}(L),\oplus,\neg_1,0,a)$ and
$(Fix_{g_{a^{\ast}}}\\
(L),\oplus,\circ_2,a^\ast,1)$ are isomorphic.
\end{proof}

\begin{theorem}\emph{ Let $L$ be an MV-algebra and $a\in B(L)$. Then MV-algebras $(Fix_{d_{a^{\ast}}}(L),\oplus,\neg_2,0,a^\ast)$ and  $(Fix_{g_{a}}(L),\oplus,\circ_1,a^\ast,1)$ are isomorphic.}
\end{theorem}
\begin{proof} For all $a\in L$, let $u:Fix_{d_{a^{\ast}}}(L)\longrightarrow Fix_{g_{a}}(L)$ be defined by $f(x)=a \oplus x$ for all $x\in Fix_{d_{a^{\ast}}}(L)$. Clearly, $u$ is a map from $Fix_{d_{a^{\ast}}}(L)$ to $Fix_{g_{a}}(L)$, that is, $u$ is well defined.

$(1)$ For all $x,y\in Fix_{d_{a^{\ast}}}(L)$, we have $u(x\oplus y)=a\oplus(x\oplus y)=a \oplus a\oplus x\oplus y=(a \oplus x)\oplus (a \oplus y)=u(x)\oplus u(y)$, that is, $u(x\oplus y)=u(x)\oplus u(y)$.

$(2)$ For all $x\in Fix_{d_{a^{\ast}}}(L)$, we have $u(\neg_2 x)=a\oplus (a\odot x^\ast)=a\vee(a\wedge x^\ast)=(a \vee a)\wedge(a\vee x)=a\oplus x$. One the other hand, we have $( u(x))^{\circ_1}=(a \oplus x)^{\circ_1}=a \oplus (a \oplus x)=a\oplus x$ and hence $u(\neg_2 x)= (u(x))^{\circ_1}$.

$(3)$ For all $x,y\in Fix_{d_{a^{\ast}}}(L)$, that is, $x=a\ominus x$ and $y=a\ominus y$, if $u(x)=u(y)$, then $a\oplus x=a\oplus y$. From $a\ominus x\leq x$, we have $x\leq a\oplus x$, then $x\leq a \oplus y$. It follows that $a\ominus x\leq y$, which implies $x\leq y$. Similarly, we can prove $y\leq x$ and hence $x=y$. Consequently, we obtain that $u$ is injective.

$(4)$  One can easily prove the fact that $x \vee y=x\oplus(y\ominus x)=y$ if and only if there exists $z\in L$ such that $y\ominus z=x$ for all $x,y\in L$. From the above fact, we will prove that $u$ is surjective. Now, for all $x\in Fix_{g_{a}}(L)$, then $x= g_{a}(x)=a\oplus x$. Using the above fact, we have that $u(x\ominus a)=a\oplus(x\ominus a)=a\oplus((a\oplus x)\ominus a)=a\oplus x=x$. Thus, we conclude that $u$ is surjective. Also, it is easy to verify that $u^{-1}(x)=x\ominus a$ is a homomorphism from  $Fix_{g_{a}}(L)$ to $Fix_{d_{a^{\ast}}}(L)$.

Combining them, we obtain that $u$ is an isomorphism from $(Fix_{d_{a^{\ast}}}(L),\oplus,\neg_2,0,a^\ast)$ to $(Fix_{g_{a}}(L),\oplus,\circ_1,a^\ast,1)$ Therefore, MV-algebras $(Fix_{d_{a^{\ast}}}(L),\oplus,\neg_2,0,a^\ast)$ and  $(Fix_{g_{a}}\\
(L),\oplus,\circ_1,a^\ast,1)$ are isomorphic.
\end{proof}

In the following, we prove a representation theorem for MV-algebras

\begin{theorem} \emph{Let $L$ be an MV-algebra and $a\in B(L)$. Then MV-algebra $L$ is isomorphic to the direct product $(Fix_{d_a}(L),\oplus,\neg_1,0,a)$ and $(Fix_{d_{a^{\ast}}}(L),\oplus,\neg_2,0,a^\ast)$.}
\end{theorem}
\begin{proof}For all $a\in L$, let $\varphi:L\longrightarrow Fix_{d_a}(L)\times Fix_{d_{a^{\ast}}}(L)$ be defined by $\varphi(x)=(x\odot a, x\ominus a)$ for all $x\in L$. Clearly, $\varphi$ is a map from $L$ to $Fix_{d_a}(L)\times Fix_{d_{a^{\ast}}}(L)$, that is, $\varphi$ is well defined.

$(1)$ From Theorem 3.13 and 4.1, we get $\varphi$ is a homomorphism from $L$ to $Fix_{d_a}(L)\times Fix_{d_{a^{\ast}}}(L)$.

$(2)$ Now, we prove that $\varphi$ is injective. If $x_1\in Fix_{d_a}(L)$ and $x_2\in Fix_{d_{a^{\ast}}}(L)$, then for $x=x_1\vee x_2$, we have $\varphi(x)=(x_1,x_2)$  since $(L,\wedge,\vee)$ is a distributive lattice and $x=(x\wedge a)\vee(x\wedge a^\ast)$ for all $x\in L$.  Consequently, we obtain that $\varphi$ is injective. Also, it is easy to verify that $\varphi^{-1}(x,y)=x\vee y$ for all $(x,y)\in Fix_{d_a}(L)\times Fix_{d_{a^{\ast}}}(L)$ is a homomorphism from $Fix_{d_a}(L)\times Fix_{d_{a^{\ast}}}(L)$ to $L$.

Combining them, we obtain that $\varphi$ is an isomorphism from $L$ to the direct product $(Fix_{d_a}(L),\oplus,\neg_1,0,a)$ and $(Fix_{d_{a^{\ast}}}(L),\oplus,\neg_2,0,a^\ast)$. Therefore, MV-algebra $L$ is isomorphic to the direct product $(Fix_{d_a}(L),\oplus,\neg_1,0,a)$ and $(Fix_{d_{a^{\ast}}}(L),\oplus,\neg_2,0,a^\ast)$.
\end{proof}

\begin{theorem} \emph{Let $L$ be an MV-algebra and $a\in B(L)$. Then MV-algebra $L$ is isomorphic to the direct product $(Fix_{g_{a^{\ast}}}(L),\oplus,\circ_2,a^\ast,1)$ and $(Fix_{g_{a}}(L),\oplus,\circ_1,a^\ast,1)$.}
\end{theorem}
\begin{proof} For all $a\in L$, let $\phi:L\longrightarrow Fix_{g_a}(L)\times Fix_{g_{a^{\ast}}}(L)$ be defined by $\varphi(x)=(x\oplus a, x\rightarrow a)$ for all $x\in L$. Clearly, $\phi$ is a map from $L$ to $Fix_{g_a}(L)\times Fix_{d_{g^{\ast}}}(L)$, that is, $\phi$ is well defined.

$(1)$ From Theorems 4.11 and Corollary 4.13, we get $\phi$ is a homomorphism from $L$ to $Fix_{g_a}(L)\times Fix_{g_{a^{\ast}}}(L)$.

$(2)$ Now, we prove that $\phi$ is injective. If $x_1\in Fix_{g_a}(L)$ and $x_2\in Fix_{g_{a^{\ast}}}(L)$, then for $x=x_1\wedge x_2$, we have $\phi(x)=(x_1,x_2)$  since $(L,\wedge,\vee)$ is a distributive lattice and $x=(x\vee a)\wedge(x\vee a^\ast)$ for all $x\in L$.  Consequently, we obtain that $\varphi$ is injective.  Also, it is easy to verify that $\phi^{-1}(x,y)=x\wedge y$ for all $(x,y)\in Fix_{g_a}(L)\times Fix_{g_{a^{\ast}}}(L)$ is a homomorphism from $Fix_{g_a}(L)\times Fix_{g_{a^{\ast}}}(L)$ to $L$.

Combining them, we obtain that $\phi$ is an isomorphism from $L$ to the direct product $(Fix_{g_{a^{\ast}}}(L),\oplus,\circ_2,a^\ast,1)$ and $(Fix_{g_{a}}(L),\oplus,\circ_1,a^\ast,1)$. Therefore, MV-algebra $L$ is isomorphic to the direct product $(Fix_{g_{a^{\ast}}}(L),\oplus,\circ_2,a^\ast,1)$ and $(Fix_{g_{a}}(L),\oplus,\circ_1,a^\ast,1)$.
\end{proof}

From Theorem 4.14-4.17, the following theorems are immediate consequence.

\begin{theorem} \emph{Let $L$ be an MV-algebra and $a\in B(L)$. Then MV-algebra $L$ is isomorphic to the direct product $(Fix_{d_a}(L),\oplus,\neg_1,0,a)$ and $(Fix_{g_{a}}(L),\oplus,\circ_1,a^\ast,1)$.}
\end{theorem}
\begin{proof} It follows from Theorems 4.14 and 4.16.
\end{proof}

\begin{theorem} \emph{Let $L$ be an MV-algebra and $a\in B(L)$. Then MV-algebra $L$ is isomorphic to the direct product $(Fix_{d_{a^{\ast}}}(L),\oplus,\neg_2,0,a^\ast)$ and $(Fix_{g_{a^{\ast}}}(L),\oplus,\circ_2,a^\ast,1)$.}
\end{theorem}

\begin{proof} It follows from Theorems 4.15 and 4.17.
\end{proof}

In what follows, as applications of Boolean derivations, some characterizations of Boolean algebras will be given.

\begin{theorem} \emph{Let $L$ be an MV-algebra. Then the following statements are equivalent:}
\begin{enumerate}[(1)]
              \item \emph{$L$ is a Boolean algebra,
              \item for any $a\in L$, Fix$_{d_a}(L)=(a]$,
              \item for any $a\in L$, Fix$_{g_a}(L)=[a)$.}
            \end{enumerate}
\end{theorem}
\begin{proof} $(1)\Rightarrow (2)$ Assume that $L$ is a Boolean algebra, we have $x\odot x=x$ for all $x\in L$. It follows that $d_a(a)=a\odot a=a$ for all $a\in L$. Thus, $a\in$ Fix$_{d_a}(L)$.  Since $L$ is a Boolean algebra, we have that Fix$_{d_a}(L)$ is an ideal of $L$ by Proposition 3.19 and Theorem 4.1. That is, for all $x\in L$, if $x\leq a$, we obtain that $x\in$ Fix$_{d_a}(L)$, which implies that $(a]\subseteq$ Fix$_{d_a}(L)$. Next, we will show that Fix$_{d_a}(L)\subseteq(a]$.  For all $x\in$ Fix$_{d_a}(L)$, we have $d_a(x)=a\odot x=a\wedge x=x$, which implies that $x\leq a$, that is, $x\in (a]$. It follows that Fix$_{d_a}(L)\subseteq(a]$. Therefore, we obtain that Fix$_{d_a}(L)=(a]$.

$(2)\Rightarrow (1)$ Assume that Fix$_{d_a}(L)=(a]$ for all $a\in L$. Since $a\in (a]$, we get $a\in$ Fix$_{d_a}(L)$. It follows that $d_a(a)=a$, that is, $a\odot a=a$ for all $a\in L$. Therefore, we obtain that $L$ is a Boolean algebra.

$(1)\Rightarrow (3)$ Assume that $L$ is a Boolean algebra, we have $x\oplus x=x$ for all $x\in L$. It follows that $g_a(a)=a\oplus a=a$ for all $a\in L$. Thus, $a\in$ Fix$_{g_a}(L)$. Since $L$ is a Boolean algebra, one can easy to check that Fix$_{g_a}(L)$ is a filter of $L$ by Theorem 4.5. That is, for all $x\in L$, if $a\leq x$, we obtain that $x\in$ Fix$_{g_a}(L)$, which implies that $[a)\subseteq$ Fix$_{g_a}(L)$. Next, we will show that Fix$_{g_a}(L)\subseteq(a]$.  For all $x\in$ Fix$_{g_a}(L)$, we have $g_a(x)=a\oplus x=a\vee x=x$, which implies that $a\leq x$, that is, $x\in [a)$. It follows that Fix$_{g_a}(L)\subseteq[a)$. Therefore, we obtain that Fix$_{g_a}(L)=[a)$.

$(3)\Rightarrow (1)$ Assume that Fix$_{g_a}(L)=[a)$ for all $a\in L$. Since $a\in [a)$, we get $a\in$ Fix$_{g_a}(L)$. It follows that $g_a(a)=a$, that is, $a\oplus a=a$ for all $a\in L$. Therefore, we obtain that $L$ is a Boolean algebra.

\end{proof}

In what follows, we focus on algebraic structure the set of all additive (implicative) Boolean derivations. We denote by $D(L)=\{d_a|a\in B(L)\} (G(L)=\{g_a|a\in B(L)\})$ be the set of all additive (implicative) Boolean derivations of $L$.

\begin{theorem}\emph{ Let $L$ be an MV-algebra. Then $(D(L),\sqcup,\sqcap,\star,d_0,d_1)$ is a Boolean algebra, where $(d_a\sqcup d_b)x=(d_ax)\vee (d_bx)$, $(d_a\sqcap d_b)x=(d_ax)\wedge (d_bx)$, $(d_a)^\star x=d_{a^\ast}x$, $(d_0)^\star=d_1$, for any $d_a,d_b\in D(L)$.}
\end{theorem}

\begin{proof} First, we show that $(D(L),\sqcup,\sqcap,d_0,d_1)$ is a bounded lattice with $d_0$ as the smallest element and $d_1$ as the greatest element. For all $d_a,d_b\in D(L)$ and $x\in L$, we have $(d_a\sqcap d_b)(x)=(d_a(x))\wedge (d_b(x))=(a\odot x)\wedge (b\odot x)=(a\wedge x)\wedge (b\wedge x)=(a\wedge b)\wedge x=(a\wedge b)\odot x=d_{a\wedge b}(x)$, that is, $(d_a\sqcap d_b)=d_{a\wedge b}$ and  hence $d_a\sqcap d_b\in D(L)$. Furthermore, we have $(d_a\sqcup d_b)(x)=(d_a(x)\vee d_b(x))=(a\odot x)\vee(b\odot x)=(a\wedge x)\vee(b\wedge x)=(a\vee b)\wedge x=(a\vee b)\odot x=d_{a\vee b}(x)$, that is, $(d_a\sqcup d_b)=d_{a\vee b}$ and hence $d_a\sqcup d_b\in D(L)$. Therefore,$(D(L),\sqcup,\sqcap,d_0,d_1)$ is a lattice. For all $d_a\in D(L)$ and $x\in L$, we have $(d_a\sqcap d_0)(x)=d_a(x)\wedge d_0(x)=0=d_0(x)$ and $(d_a\sqcup d_1(x)=d_a(x)\vee d_1(x)=x=d_1(x)$. Thus, $d_0$ is the smallest element and $d_1$ is the greatest element in $D(L)$.

Next, we prove that $(D(L),\sqcup,\sqcap)$ is a distributive lattice. For all $d_a,d_b,d_c$ and $x\in L$, from $(L,\vee,\wedge)$ is a distributive lattice, one can prove $(d_a\sqcup(d_b\sqcap d_c)=(d_a\sqcup d_b)\sqcap (d_a\sqcup d_c)$ and $(d_a\sqcap(d_b\sqcup d_c)=(d_a\sqcap d_b) \sqcup (d_a\sqcap d_c)$. Therefore, $(D(L),\sqcup,\sqcap,d_0,d_1)$ is a bounded distributive lattice.

Finally, we prove that $(D(L),\sqcup,\sqcap,\star,d_0,d_1)$ is a Boolean algebra. For all $d_a\in D(L)$ and $x\in L$, we have $(d_a)^\star(x)=d_{a^\ast}(x)=a^\ast\odot x=d_{a^\ast}(x)$, that is, $(d_a)^\star(x)=d_{a^\ast}(x)$. One can easy check that if $a\in B(L)$, then $a^\ast\in B(L)$, and hence ${d_a}^\star\in D(L)$. furthermore, we have $(d_a\sqcup (d_a)^\star)(x)=(d_a)(x)\vee d_{a^\ast}(x)=(a\odot x)\vee(a^\ast\odot x)=(a\vee a^\ast)\odot x=x=d_1(x)$ and $(d_a\sqcap (d_a)^\star)(x)=(d_a)(x)\wedge d_{a^\ast}(x)=(a\odot x)\wedge(a^\ast\odot x)=(a\vee a^\ast)\wedge x=0=d_0(x)$, that is, $(d_a\sqcup (d_a)^\star)=d_1$ and $(d_a\sqcap (d_a)^\star)=d_0$. Therefore, $(D(L),\sqcup,\sqcap,\star,d_0,d_1)$ is a Boolean algebra.
\end{proof}

\begin{theorem}\emph{ Let $L$ be an MV-algebra. Then $(G(L),\cap,\cup,\bullet,g_0,g_1)$ is a Boolean algebra, where $(g_a\cup g_b)x=(g_ax)\vee (g_bx)$, $(g_a\cap g_b)x=(g_ax)\wedge (g_bx)$, $(g_a)^\bullet x=g_{a^\ast}x$, $(g_0)^\bullet=g_1$, for any $g_a,g_b\in G(L)$.}
\end{theorem}

\begin{proof} The proof is similar to that of Theorem 4.21.
\end{proof}
\begin{theorem}\emph{ Let $L$ be an MV-algebra. Then Boolean algebras $(B(L),\wedge,\vee,\ast,0,1)$ and $(D(L),\sqcup,\sqcap,\star,d_0,d_1)$ are isomorphic.}
\end{theorem}

\begin{proof}For all $a\in B(L)$, let $\phi:B(L)\longrightarrow D(L)$ be defined by $\phi(a)=d_a$ for all $a\in B(L)$. Clearly, $\phi$ is a map from $ B(L)$ to $D(L)$, that is, $\phi$ is well defined. One can easily see that $\phi$ is one to one and onto.

Furthermore, for any $a,b\in B(L)$, we have $\phi(a\wedge b)=d_{a\wedge b}=d_a\sqcap d_b=\phi(a)\sqcap \phi(b)$, $\phi(a\vee b)=d_{a\vee b}=d_a\sqcup d_b=\phi(a)\sqcup \phi(b)$ and $\phi(a^\ast)=d_{a^\ast}=(d_a)^\star=(\phi(a))^\star$.

Combining them, we obtain that $\phi$ is an isomorphism from $B(L)$ to $D(L)$. Therefore, Boolean algebras $(B(L),\wedge,\vee,\ast,0,1)$ and $(D(L),\sqcup,\sqcap,\star,d_0,d_1)$ are isomorphic.
\end{proof}

\begin{theorem}\emph{ Let $L$ be an MV-algebra. Then Boolean algebras $(B(L),\wedge,\vee,\ast,0,1)$ and $(G(L),\cap,\cup,\bullet,g_0,g_1)$ are isomorphic.}
\end{theorem}
\begin{proof} The proof is similar to that of Theorem 4.23.
\end{proof}

The next theorem is now an immediate consequence.

\begin{theorem}\emph{Let $L$ be a Boolean algebra. Then}
\begin{enumerate}[(1)]
  \item \emph{$(L,\wedge,\vee,\ast,0,1)$ is  isomorphic to $(D(L),\sqcup,\sqcap,\star,d_0,d_1)$,
  \item $(L,\wedge,\vee,\ast,0,1)$ is  isomorphic to $(G(L),\cap,\cup,\bullet,g_0,g_1)$.}
\end{enumerate}
\end{theorem}

\begin{proof} It follows from Theorem 4.23, 4.24.
\end{proof}

\section{Conclusions}

The notion of derivations is helpful for studying structures and properties in algebraic systems. In the paper, some useful properties of particular derivations are discussed. Also, we obtain that the fixed point set of additive derivations is still an MV-algebra. Besides, we get that the fixed point set of Boolean additive derivations and that of their adjoint derivations are isomorphism. Finally, we obtain that the set of all Boolean additive  derivations is  isomorphic to a Boolean algebra. There is still an open problem: for any ideal $I$ of a general MV-algebra $L$, whether there exists an additive derivation $d$ such that Fix$_d(L)=I$. In our future work, we will consider these problems.

\medskip
\noindent\textbf{Acknowledgments}
\medskip

\indent This research is partially supported by a grant of
National Natural Science Foundation of China (11601302).

%%%%%%%%%%%%%%%%%%%%%%%%%%%%%%%%%%%%%%%%%%%%%%%%%%%%%%%%%%%%%%%%%%%%%%%%%%%%%%%%%%%%%%%%%%%%%%%%%%%%%%%%%%%%%%%%%%%%%%%%%%%%%%%%%%%%%%%%%%%%%%%%
%%%%%%%%%%%%%%%%%%%%%%%%%%%%%%%%%%%%%%%%%%%%%%%%%%%%%%%%%%%%%%%%%%%%%%%%%%%%%%%%%%%%%%%%%%%%%%%%%%%%%%%%%%%%%%%%%%%%%%%%%%%%%%%%%%%%%%%%%%%%%%%%


\begin{thebibliography}{00}
\bibitem{Albas} E. Albas, On ideals and orthogonal generalized derivations of semiprime ring, Math. J. Okayama Univ., {\bf 49} (2007) 53-58.
\bibitem{Alshehri} N. O. Alshehri, Derivations of MV-algebras, Int. J. Math. Math. Sci., 2010, doi:10.1155/2010/312027.
\bibitem{Davvaz2} L. K. Ardekani, B. Davvaz, On Generalized derivations of BCI-algebras and their properties, Journal of Mathematics, {\bf 2015} (2015) 1-10.
\bibitem{Ardekani} L. K. Ardekani, B. Davvaz, $f$-derivations and $(f,g)$-derivations of MV-algebras, J. algebraic syst., {\bf 1} (2013) 11-31.
\bibitem{Bell1} H. E. Bell, L. C. Kappe, Rings in which derivations satisfy certain algebraic conditions, Acta Math. Hungar., {\bf 53} (1989) 339-346.
\bibitem{Bell2} H. E. Bell, G. Mason, On derivations in near-rings and near-fields, North-Holland Math. Studies., {\bf 137} (1987) 31-35.
\bibitem{Borzooei} R. A. Borzooei,  O. Zahiri, Some results on derivations of BCI-algebras, J. natur. sci. math., {\bf 26} (2013) 529-545.
\bibitem{Chang1} C. C. Chang, Algebraic analysis of many-valued logic, Trans Am Math Soc., {\bf88} (1958) 467-490.
\bibitem{Chang2} C. C. Chang, A new proof of the completeness of the {\L}ukasiewicz axioms, Trans Am Math Soc., {\bf93} (1959) 74-80.
\bibitem{Mundici1} R. Cignoli, D. Mundici, An elementary proof of Chang's completeness theorem for the infinite-valued calculus of {\L}ukasiewicz, Stud. Log., {\bf58} (1997) 79-97.
\bibitem{Mundici2} R. Cignoli, I.M.L. D'Ottaviano, D. Mundici, Algebraic Foundations of Many-valued Reasoning. Kluwer Academic Publ., Dordrecht, 2000.
\bibitem{Davvaz1} B. Davvaz, Generalized derivations of rings and Banach algebras, Comm. Algebra., {\bf 41} (2013) 1188-1194.
\bibitem{Ferrari} L. Ferrari, On derivations of lattices, Pure Math Appl., {\bf 12} (2001) 365-382.
\bibitem{Ghorbain} S. Ghorbain, L. Torkzadeh, S. Motamed, $(\odot,\oplus)$-derivations and $(\ominus,\odot)$-derivations on MV-algebras, Iran. J. Math. Sci. Inform., {\bf 8} (2013) 75-90.
\bibitem{Gratzer} G. Gr\"{a}tzer, Lattice theory, W. H. Freeman and Company, San Francisco, 1979.
\bibitem{he} P. F. He, X. L. Xin, J. M. Zhan, On derivations and their fixed point sets in residuated lattices, Fuzzy Sets Syst., {\bf 303} (2016) 97-113.
\bibitem{Jun} Y. B. Jun, X. L. Xin, On derivations of BCI-algebras, Inform. Sci., {\bf 159} (2004) 167-176.
\bibitem{Tarski} J. {\L}ukasiewicz, A. Tarski, Untersuchungen \"{u}ber den Aussagenkalk\"{u}l, Comptes Rendus des S\'{e}ances de la Soci\'{e}t\'{e} des Science et des Letters de., {\bf23} (1930) 30-50.
\bibitem{Posner} E. Posner, Derivations in prime rings, Proc. Amer. Math. Soc., {\bf 8} (1957) 1093-1100.
\bibitem{Turunen} E. Turunen, Mathematics Behind Fuzzy Logic. Physica-Verlag, Heidelberg, 1999.
\bibitem{Xin} X. L. Xin, M. Feng, Y. W. Yang, On $\odot$-derivations of BL-algebras, J. of  Math (P. R. C) ., {\bf 36} (2016) 552-558.
\bibitem{Yazarli} H. Yazarli, A note on derivations in MV-algebras, Miskolc Math Notes., {\bf 14} (2013) 345-354.
\bibitem{Zhan} J. M. Zhan, Y. L. Liu, On $f$-derivations of BCI-algebras, Int. J. Math. Math. Sci., {\bf 11} 2005, 176-191.







\end{thebibliography}
\end{document}